\def\marker{\>\hbox{${\vcenter{\vbox{
    \hrule height 0.4pt\hbox{\vrule width 0.4pt height 6pt
    \kern6pt\vrule width 0.4pt}\hrule height 0.4pt}}}$}\>}

\documentclass[12pt]{article}
\usepackage{array,amsmath,amssymb,amsthm}  				    
\usepackage{fullpage,lineno}
\usepackage{graphicx}
\usepackage{color}

\newtheorem{theorem}{Theorem}[section]

\newtheorem{conjecture}[theorem]{Conjecture}

\newtheorem{problem}[theorem]{Problem}
\newtheorem{lemma}[theorem]{Lemma}

\newtheorem{corollary}[theorem]{Corollary}

\theoremstyle{definition}
\newtheorem{definition}[theorem]{Definition}
\newtheorem{example}[theorem]{Example}
\newtheorem{remark}[theorem]{Remark}

\def\nul{\varnothing} 
\def\VEC#1#2#3{#1_{#2},\ldots,#1_{#3}}

\def\SE#1#2#3{\sum_{#1=#2}^{#3}}

\def\CH#1#2{\binom{#1}{#2}} 

\def\FL#1{\left\lfloor{#1}\right\rfloor} 
\def\CL#1{\left\lceil{#1}\right\rceil}   

\def\NN{{\mathbb N}}

\def\C#1{\left | #1 \right |}    

\def\cD{{\mathcal D}}
\def\cF{{\mathcal F}}
\def\bU{{\overline U}}
\def\esub{\subseteq}

\def\eps{\varepsilon}

\long\def\skipit#1{}

\begin{document}

\title{Acyclic graphs with at least $2\ell+1$ vertices
are $\ell$-recognizable}

\author{
Alexandr V. Kostochka\thanks{University of Illinois,
Urbana, IL:
\texttt{kostochk@math.uiuc.edu}.  Supported by NSF
grants DMS-1600592 and DMS-2153507.}\,,
Mina Nahvi\thanks{University of Illinois,
Urbana, IL: \texttt{mnahvi2@illinois.edu}.}\,,
Douglas B. West\thanks{Zhejiang Normal Univ., Jinhua, China
and University of Illinois, Urbana, IL:
\texttt{dwest@illinois.edu}.  Supported by National Natural Science Foundation
of China grants NSFC 11871439, 11971439, and U20A2068.}\,,
Dara Zirlin\thanks{University of Illinois at Urbana--Champaign, Urbana IL 61801:
\texttt{zirli22d@mtholyoke.edu}.  Supported by Arnold O. Beckman Campus
Research Board Award RB20003 of the University of Illinois 
and by NSF RTG grant DMS-1937241.}
}

\date{\today}
\maketitle

\baselineskip 16pt

\begin{abstract}
The {\it $(n-\ell)$-deck} of an $n$-vertex graph is the multiset of subgraphs
obtained from it by deleting $\ell$ vertices.  A family of $n$-vertex graphs is
{\it $\ell$-recognizable} if every graph having the same $(n-\ell)$-deck as a
graph in the family is also in the family.  We prove that the family of 
$n$-vertex graphs with no cycles is $\ell$-recognizable when $n\ge2\ell+1$
(except for $(n,\ell)=(5,2)$).  As a consequence, the family of $n$-vertex
trees is $\ell$-recognizable when $n\ge2\ell+1$ and $\ell\ne2$.  It is known
that this fails when $n=2\ell$.
\end{abstract}

\section{Introduction}

The {\it $j$-deck} of a graph is the multiset of its $j$-vertex induced
subgraphs.  We write this as the $(n-\ell)$-deck when the graph has $n$
vertices and the focus is on deleting $\ell$ vertices.  An $n$-vertex graph is
{\it $\ell$-reconstructible} if it is determined by its $(n-\ell)$-deck.
Since every member of the $(j-1)$-deck arises $n-j+1$ times by deleting a 
vertex from a member of the $j$-deck, the $j$-deck of a graph determines its
$(j-1)$-deck.  Therefore, a natural reconstruction problem is to find for each
graph the maximum $\ell$ such that it is $\ell$-reconstructible.  For this
problem, Manvel~\cite{M69,M74} extended the classical Reconstruction
Conjecture of Kelly~\cite{Kel1} and Ulam~\cite{U}.

\begin{conjecture}[{\rm Manvel~\cite{M69,M74}}]
For $\ell\in\NN$, there exists a threshold $M_\ell$ such that every graph with
at least $M_\ell$ vertices is $\ell$-reconstructible.
\end{conjecture}

\noindent
Manvel named this ``Kelly's Conjecture'' in honor of the final sentence in
Kelly~\cite{Kel2}, which suggested that one can study reconstruction from
the $(n-2)$-deck.  Manvel noted that Kelly may have expected the statement to
be false.

Many reconstruction arguments have two parts.  First, one proves that the deck
determines that the graph is in a particular class or has a particular property.
When the $(n-\ell)$-deck determines this, the property is
{\it $\ell$-recognizable}.  Separately, using the knowledge that the deck
determines whether the graph has that property, one proves that only one graph
with that property has that deck.  This makes the family {\it weakly
$\ell$-reconstructible}, meaning that no two graphs in the family have the same
deck.  Bondy and Hemminger~\cite{BH} articulated the distinction between
these two steps for the case $\ell=1$.

Here, toward $\ell$-reconstructibility of trees, we consider
$\ell$-recognizability of acyclic graphs.  We prove the following theorem.

\begin{theorem}\label{mainthm}
For $n\ge2\ell+1$, except when $(n,\ell)=(5,2)$, the family of $n$-vertex
acyclic graphs is $\ell$-recognizable.
\end{theorem}

We forbid $(n,\ell)=(5,2)$ due to the two graphs in Figure~\ref{52ex}, which
have the same $3$-deck.  Indeed, this possibility must be excluded from many of
the claims we prove.

\begin{figure}[h]
\begin{center}
\includegraphics[scale=0.5]{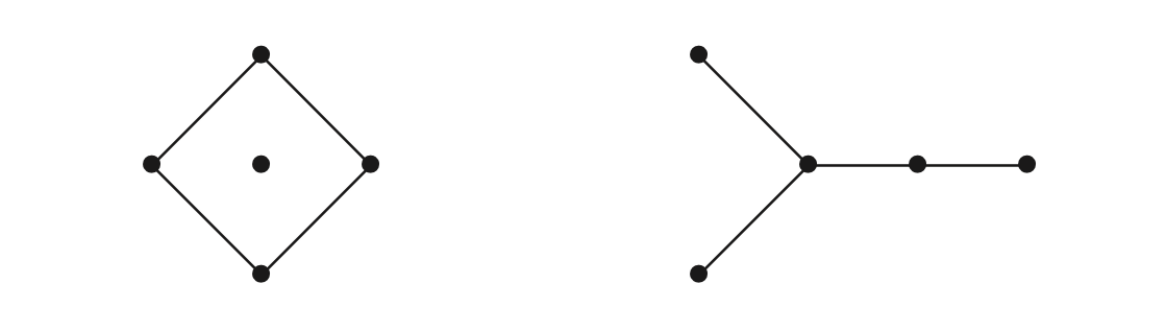}
\caption{$5$-vertex graphs with the same $3$-deck}\label{52ex}
\end{center}
\end{figure}

\vspace{-1pc}

Since the $(n-\ell)$-deck determines the $2$-deck when $n-\ell\ge2$, in this
setting we also know the number of edges.  This yields the following corollary.

\begin{corollary}
For $n\ge2\ell+1$, except when $(n,\ell)=(5,2)$, the family of $n$-vertex trees
is $\ell$-recognizable.
\end{corollary}

Spinoza and West~\cite{SW} determined, for every graph $G$ with maximum degree
at most $2$, the maximum $\ell$ such that $G$ is $\ell$-reconstructible.  Their
full result is complicated to state, but a special case is that for
$n\ge2\ell+1$ (except $(n,\ell)=(5,2)$), every $n$-vertex graph with
maximum degree at most $2$ is $\ell$-reconstructible.  They also show that 
a path with $2\ell$
vertices has the same $\ell$-deck as the disjoint union of an $(\ell+1)$-cycle
and a path with $\ell-1$ vertices, so the threshold
$n\ge2\ell+1$ in both~\cite{SW} and Theorem~\ref{mainthm} is sharp.

N\'ydl~\cite{N90} conjectured that trees with at least $2\ell+1$ vertices are
weakly $\ell$-reconstructible.  This conjecture would be sharp, since
N\'ydl~\cite{N81} found two trees with $2\ell$ vertices having the same
$\ell$-deck.  The two trees are obtained from a path with $2\ell-1$ vertices by
adding one leaf, adjacent either to the central vertex of the path or to one of
its neighbors.  Kostochka and West~\cite{KW} used the results of~\cite{SW} to
give a short proof of this result of N\'ydl.  However, one counterexample to
N\'ydl's conjecture is known: Groenland, Johnston, Scott, and Tan~\cite{GJST}
obtained two $13$-vertex trees having the same $7$-deck.  Excluding this
example and incorporating the $\ell$-recognizability of trees leads to a
modification of N\'ydl's conjecture.

\begin{conjecture}
For $n\ge2\ell+1$, except when $(n,\ell)\in\{(5,2),(13,6)\}$, every $n$-vertex
tree is $\ell$-reconstructible.  The threshold on $n$ is known to be sharp.
\end{conjecture}

For $\ell=2$, the threshold on $n$ must be at least $6$ due to the graphs 
in Figure~\ref{52ex}.  Giles~\cite{Gil}
proved that trees with at least six vertices are $2$-reconstructible (using
only the connected members of the deck).  For general $\ell$, Groenland et
al.~\cite{GJST} proved that $n\ge9\ell+24\sqrt{2\ell}+o(\sqrt\ell)$ suffices
for $\ell$-reconstructibility of $n$-vertex trees.  In~\cite{KNWZ}, the
present authors prove that $n\ge6\ell+11$ suffices.

Besides acyclicity, another fundamental property of trees is connectedness.
Spinoza and West~\cite{SW} proved that connectedness is $\ell$-recognizable for
$n$-vertex graphs when $n>2\ell^{(\ell+1)^2}$.  Later, Groenland et
al.~\cite{GJST} reduced the general threshold to $n\ge 10\ell$.
Manvel~\cite{M74} proved that connectedness is $2$-recognizable
for graphs with at least six vertices, and the present authors~\cite{KNWZ2}
proved that connectedness is $3$-recognizable for graphs with at least seven
vertices.  Spinoza and West~\cite{SW} suggested that (except for
$(n,\ell)=(5,2)$), connectedness is recognizable for $n$-vertex graphs when
$n\ge2\ell+1$.

The $(n-\ell)$-deck of a graph is {\it acyclic} if each card in the deck is
acyclic.  As a step toward the threshold on $n$ for $\ell$-recognizability of 
connectedness, one can consider the special case of $n$-vertex graphs whose
$(n-\ell)$-deck is acyclic.  Our result in this paper settles the question for
graphs with $n-1$ edges, where connectedness and acyclicity are equivalent
(the number of edges is known from the $2$-deck).  This suggests other detailed
questions.

\begin{problem}
For $\ell\ge1$ and $c\ge0$, determine the smallest thresholds $N_{\ell,c}$ and
$N'_{\ell,c}$ such that for all $n$-vertex graphs with $n+c$ edges whose
$(n-\ell)$-deck $\cD$ is acyclic,

(a) if $n\ge N_{\ell,c}$, then $\cD$ determines whether the graph
is connected, and

(b) if $n\ge N'_{\ell,c}$, then the graph is connected.

\noindent
The thresholds when the deck is not required to be acyclic are also unknown.
\end{problem}

We note that $N'_{\ell,1}=2\ell$.
For the upper bound, consider a disconnected $n$-vertex graph with an acyclic
$(n-\ell)$-deck, where $n\ge2\ell$.  A smallest component must be acyclic,
since a cycle would have length at most $n/2$ and be seen in the deck.  Hence
some other component $H$ must have at least $\C{V(H)}+2$ edges.  However, a
$p$-vertex graph with at least $p+2$ edges has girth at most $\FL{(p+2)/2}$
(see Exercise 5.4.36 of~\cite{West}, for example), yielding a cycle in a card.
For the lower bound, we seek a disconnected graph with $2\ell-1$ vertices whose
$(\ell-1)$-deck is acyclic.  When $\ell$ is even, the graph consists of
an isolated vertex plus four paths of length $\ell/2$ with common endpoints.
When $\ell$ is odd, the nontrivial component consists of a cycle of length
$2\ell-2$ plus two diametric chords creating cycles of length $\ell$ (this
example was contributed by a referee).  It is possible that the threshold
$N_{\ell,1}$ for determining connectedness from the $(n-\ell)$-deck is smaller.

For $c=0$, we conjecture $N_{\ell,0}=2\ell-1$.
The lower bound holds because a $(2\ell-2)$-cycle and the
disjoint union of two $(\ell-1)$-cycles have the same
$(\ell-2)$-deck.  Zirlin~\cite{Z} proved $N_{\ell,0}\le 2\ell+1$ for
$\ell\ge3$, and she proved $N_{\ell,0}=2\ell-1$ for $\ell\ge45$.

In Section~\ref{Stools} we develop tools that are useful for reconstructing
information from acyclic decks.  In Section~\ref{S2l2} we prove that
$n\ge2\ell+2$ suffices for $\ell$-recognizability of $n$-vertex acyclic graphs.
In Section~\ref{S2l1} we obtain the sharp threshold, $2\ell+1$.

\section{Vines, Diameter, and Marking}\label{Stools}

Let $\cD$ be the $(n-\ell)$-deck of an $n$-vertex graph $G$ (we henceforth just
call it the ``deck'').  We will assume $n>2\ell$.  The members of $\cD$ are the
``cards'' in the deck.

\begin{definition}
The {\it eccentricity} $\eps_G(v)$ of a vertex $v$ in a graph $G$ is the
maximum of the distances from $v$ to other vertices.  The {\it radius} is
$\min_{v\in V(G)}\eps_G(v)$, and the {\it diameter} is
$\max_{v\in V(G)}\eps_G(v)$.
A {\it center} of $G$ is a vertex of minimum eccentricity.

In $G$, the {\it $j$-ball} at a vertex $v$ is the
subgraph induced by all vertices within distance $j$ of $v$ in $G$.  The
{\it $j$-eball} at an edge $e$ is the subgraph induced by all vertices within
distance $j$ of either endpoint of $e$.  A {\it $j$-vine} or {\it $j$-evine} is
a tree having diameter $2j$ or $2j+1$, respectively.  A {\it $j$-center} is
a vertex that is the center of a $j$-vine; a {\it $j$-central edge} is the
central edge of a $j$-evine (joining the two centers).
\end{definition}

The term ``$j$-vine'' follows the botanical theme in terminology about trees;
a vine grows from its main path.  Note that if the $j$-ball at a vertex $v$
in a graph $G$ is a tree but does not contain a path with $2j+1$ vertices,
then $v$ is not a $j$-center.  When $v$ {\it is} a $j$-center, the maximal
$j$-vine with center $v$ is just the $j$-ball at $v$.

We will be interested in $j$-vines and $j$-evines that are induced subgraphs
of every reconstruction from the given deck.  Our aim is to consider an acyclic
and a nonacyclic graph having the same deck, show that they have the same
number of $j$-centers for an appropriate value $j$, and obtain a contradiction
by showing that they cannot have the same number of $j$-centers.  The key
property that will permit counting the $j$-centers in a reconstruction is in
the next lemma; it implies (under the girth condition) that maximal $j$-vines
correspond bijectively to centers of $j$-vines (that is, $j$-centers),
and similarly for $j$-evines.

\begin{lemma}\label{vinemax}
In a graph $G$ with girth at least $2j+2$, every $j$-vine lies in a unique
maximal $j$-vine.  If $G$ has girth at least $2j+3$, then every $j$-evine lies
in a unique maximal $j$-evine.
\end{lemma}
\begin{proof}
When $H$ is a $j$-vine or a $j$-evine, let $H'$ be the $j$-ball or $j$-eball in
$G$ at the center(s) of $H$, respectively.  All vertices in any $j$-vine or
$j$-evine containing $H$ lie in $H'$.  Thus $H'$ is the desired unique maximal
object unless it contains a cycle.

Let $Q$ be a shortest cycle in $H'$.  Because the vertex or vertices on $Q$
that are farthest from the center of $H'$ have distance at most $j$ from the
center, $Q$ has at most $2j+1$ vertices if $H'$ has a unique center and at most
$2j+2$ vertices if $H'$ has a central edge, contradicting the hypothesis on the
girth of $G$.
\end{proof}

\begin{example}
To see that the girth condition in Lemma~\ref{vinemax} is sharp, let
$G$ be a graph consisting of a $(2j+1)$-cycle $Q$ plus two paths of length $j$
grown from a single vertex $v$ on $Q$.  Deleting from $G$ the two vertices of
$Q$ that are farthest from $v$ yields a $j$-vine $H$.  Replacing either one of
those vertices yields a maximal $j$-vine in $G$ containing $H$, so the maximal
$j$-vine containing $H$ is not unique.  An analogous example for $j$-evines
consists of a $(2j+2)$-cycle plus paths of length $j$ grown from two
consecutive vertices.
\end{example}

\begin{definition}
Given a family $\cF$ of graphs, an {\it $\cF$-subgraph} of a graph $G$ is an
induced subgraph of $G$ belonging to $\cF$.  Let $s(F,G)$ denote the number of
occurrences of $F$ as an induced subgraph of $G$.  Let $m(F,G)$ be the number
of occurrences of $F$ as a maximal $\cF$-subgraph in $G$ (with respect to
induced subgraphs).
\end{definition}

The special case $\ell=1$ of the next lemma is due to Greenwell and
Hemminger~\cite{GH}.  Similar statements for general $\ell$ appear for 
example in~\cite{GJST}.  We include a proof for completeness; it is slightly
simpler than proofs in the literature involving inclusion chains of subgraphs.

\begin{lemma}\label{counting}
Fix an $n$-vertex graph $G$, and let $\cF$ be a family of graphs such that
every $\cF$-subgraph of $G$ lies in a unique maximal $\cF$-subgraph of $G$.
If the value of $m(F,G)$ is known for every $F\in\cF$ with at least $n-\ell$
vertices, then for all $F\in \cF$ the $(n-\ell)$-deck of $G$ determines
$m(F,G)$.
\end{lemma}
\begin{proof}
Let $t=\C{V(G)}-\C{V(F)}$; we use induction on $t$.  When $t\le\ell$, the value
$m(F,G)$ is given.  When $t>\ell$, group the induced subgraphs of $G$
isomorphic to $F$ according to the unique maximal $\cF$-subgraph of $G$
containing them (as an induced subgraph).  Counting all copies of $F$ then
yields
$$s(F,G)=\sum_{H\in{\cF}} s(F,H)m(H,G).$$
Since $\C{V(F)}<n-\ell$, we know $s(F,G)$ from the deck, and we know
$s(F,H)$ when $F$ and $H$ are known.  By the induction hypothesis, we know all
values of the form $m(H,G)$ when $F$ is an induced subgraph of $H$ except
$m(F,G)$.  Therefore, we can solve for $m(F,G)$.
\end{proof}

Before continuing with preparation for $\ell$-recognizability of acyclicity,
we note one application of Lemma~\ref{counting} that was stated incorrectly in
the paper by Kostochka and West~\cite{KW}; it also illustrates the technique
we use with $j$-vines.  The special case for $\ell=1$ was observed by
Kelly~\cite{Kel2} using different methods.
Let $P_n$ and $C_n$ respectively denote a path and a cycle with $n$ vertices,
and let $G+H$ denote the disjoint union of graphs $G$ and $H$.

\begin{corollary}
If $n>2\ell$, then $n$-vertex graphs having no component with more than
$n-\ell$ vertices are $\ell$-reconstructible, and this threshold on $n$ is 
sharp.  All $n$-vertex graphs having no component with at least $n-\ell$
vertices are $\ell$-reconstructible, with no restriction on $n$.
\end{corollary}
\begin{proof}
Let $\cF$ be the family of connected graphs; $\cF$ satisfies the
property stated in the first sentence of Lemma~\ref{counting} for any $G$.

Now consider $m(F,G)$ for $F\in\cF$.  When $n>2\ell$, an $n$-vertex graph has
at most one component with at least $n-\ell$ vertices, and it has no component
with more vertices if and only if it has at most one connected $(n-\ell)$-card.
Hence the hypothesized condition here is $\ell$-recognizable.  If some
component has exactly $n-\ell$ vertices, then it is seen as a card.  Hence
$m(F,G)$ is known for $F\in\cF$ with at least $n-\ell$ vertices, and by using
Lemma~\ref{counting} we obtain all the components of $G$ under either 
hypothesis in the statement.

The result is sharp, since $P_{\ell}+P_{\ell}$ and $P_{\ell+1}+P_{\ell-1}$ have
the same $\ell$-deck.  This follows from the result of Spinoza and
West~\cite{SW} that any two graphs with the same number of vertices and edges
whose components are all cycles with at least $j+1$ vertices or paths with
at least $j-1$ vertices have the same $j$-deck.
\end{proof}

When we speak of $j$-vines and $j$-evines in an $n$-vertex graph $G$,
we always consider only induced subgraphs.  For a given graph $G$,
a particular value of $j$ determined by the $(n-\ell)$-deck of $G$ will be of
interest.  Recall that we require $n\ge2\ell+1$ and $\ell\ge1$, so
$n-\ell\ge2$.

\begin{definition}\label{defk}
For a given graph $G$, let $k$ denote the largest integer such that $G$
contains a $k$-evine and, for $0\le j\le k$, every $j$-evine and every $j$-vine
in $G$ has fewer than $n-\ell$ vertices.  Since every edge is a $0$-evine,
$k$ is well-defined.  {\bf This fixes $k$ in terms of $G$ for the
remainder of the paper.}
\end{definition}

We consider $n$-vertex reconstructions from an acyclic $(n-\ell)$-deck $\cD$.

\begin{lemma}\label{kprop}
The value $k$ is determined by the deck $\cD$ of $G$.  That is, all
reconstructions from $\cD$ have the same value of $k$.
\end{lemma}
\begin{proof}
All subgraphs of $G$ having at most $n-\ell$ vertices are visible in $\cD$.
The deck thus gives a candidate for $k$.  Let $k'$ be the largest integer such
that some induced subgraph of a card is a $k'$-evine and, for $0\le j\le k'$,
every $j$-evine and every $j$-vine that appears in a card has fewer then
$n-\ell$ vertices.  Since this condition on $k'$ is satisfied by $k$, we have
$k'\ge k$.

The value $k'$ is strictly greater than $k$ if and only if $G$ contains a
$j$-evine or $j$-vine $R$ with at least $n-\ell$ vertices for some $j$
at most $k'$.  By the definition of $k'$, this $R$ is not contained in a card
and has strictly more than
$n-\ell$ vertices.  Without modifying a fixed longest path $P$, we can trim $R$
to $n-\ell$ vertices by iteratively deleting leaves outside $P$ unless we still
have more than $n-\ell$ vertices when only $P$ remains.  Therefore, since $G$
has no $j$-evine or $j$-vine with exactly $n-\ell$ vertices (by the 
definition of $k'$), we have $2j+1>n-\ell$.

On the other hand, the definition of $k'$ gives us a $k'$-evine contained in a
card and having fewer than $n-\ell$ vertices.  A longest path in this subgraph 
has $2k'+2$ vertices, so $2k'+2<n-\ell$.  We thus have
$$
2k'+3\le n-\ell\le 2j\le 2k'.
$$
This contradiction implies that $k'$ must equal $k$.
\end{proof}

\begin{lemma}\label{girth}
Every reconstruction from $\cD$ has girth at least $2k+4$.
\end{lemma}
\begin{proof}
The claim holds for a reconstruction $G$ having no cycle, so suppose that
$G$ has a cycle.  Since $\cD$ is acyclic, the girth of $G$ is at least
$n-\ell+1$.  Deleting some consecutive vertices from a shortest cycle yields an
induced path $R$ with $n-\ell$ vertices.  Let $t=\FL{(n-\ell-1)/2}$, so
$n-\ell+1\ge2t+2$.  The path $R$ is a $t$-vine (if $n-\ell$ is odd) or a
$t$-evine (if $n-\ell$ is even) with $n-\ell$ vertices.  The definition of $k$
thus requires $k<t$.  We compute $n-\ell+1\ge 2t+2\ge2k+4$.  Thus $G$ has girth
at least $2k+4$.
\end{proof}

\begin{corollary}\label{countk}
For $j\le k$, the deck $\cD$ determines the maximal $j$-evines and maximal
$j$-vines, with multiplicity.  Also, all reconstructions from $\cD$ have the
same numbers of $j$-centers and $j$-central edges.
\end{corollary}
\begin{proof}
Fix $j$ with $j\le k$, and let $\cF$ be the family of $j$-vines or the family
of $j$-evines.  By Lemma~\ref{kprop}, all reconstructions have the same value
of $k$.  By the definition of $k$, we obtain $m(T,G)=0$ whenever $G$ is a
reconstruction from $\cD$ and $T$ is a member of $\cF$ having at least $n-\ell$
vertices.  Since $G$ has girth at least $2k+4$ (by Lemma~\ref{girth}), every
member of $\cF$ lies in a unique maximal member of $\cF$ (by
Lemma~\ref{vinemax}).  With these properties, Lemma~\ref{counting} applies to
compute $m(T,G)$ for all $T\in\cF$.

With girth at least $2k+4$, there is also a one-to-one correspondence between
the maximal $j$-vines and the $j$-centers, and similarly for the maximal
$j$-evines and $j$-central edges.  Thus we obtain the total number of
$j$-centers and the total number of $j$-central edges.
\end{proof}

Setting $j=1$ in Corollary~\ref{countk} almost provides the degree list.
Groenland et al.~\cite{GJST} proved the strong result that the degree
list is $\ell$-reconstructible for all $n$-vertex graphs whenever
$n-\ell>\sqrt{2n\log(2n)}$.  Taylor~\cite{Tay} had shown that 
$n>f(\ell)$ suffices, where $f$ is a particular function such that $f(\ell)$
is asymptotic to ${\rm e}\ell$.  For the context of acyclic decks we obtain a
simpler intermediate threshold.  For a vertex $v$ in a graph, let $N[v]$
denote the closed neighborhood of $v$ (the set of vertices equal or adjacent
to $v$).

\begin{corollary}\label{degrees}
For $n\ge2\ell+1$ with $(n,\ell)\ne (5,2)$, the degree list of any $n$-vertex
graph with an acyclic $(n-\ell)$-deck is determined by its deck.
\end{corollary}
\begin{proof}
The case $\ell=1$ is well known: subtract the number of edges in each card from
the total number of edges.  Since $(n,\ell)\ne (5,2)$ and $n\ge 2\ell+1$, we
may thus assume $n-\ell\ge4$.

Let $G$ be an $n$-vertex graph with an acyclic $(n-\ell)$-deck $\cD$.
Since $\cD$ is acyclic and $n-\ell\ge4$, all stars are induced subgraphs.
Those with at least three vertices are the $1$-vines.

Call a vertex $v\in V(G)$ {\em big} if $d_G(v)\geq n-\ell-1$.
A vertex with degree at least $2$ in $G$ is the center of a maximal $1$-vine.
For $t\ge2$ the number of vertices with degree $t$ is the number of maximal
$1$-vines with $t+1$ vertices.  Lemma~\ref{counting} provides these values for
$t\ge2$ if we know the number of big vertices with each degree.
There are no big vertices in $G$ if and only if no card is a star.

Since $n-\ell\ge4$ and $\cD$ is acyclic,
\begin{equation}\label{acdeck}
\mbox{\em $G$ has no $3$-cycles or $4$-cycles.}
\end{equation}
If $G$ has exactly one big vertex and its degree is $d$, then
exactly $\CH d{n-\ell-1}$ cards are stars.

Suppose that $x$ and $y$ are distinct big vertices in $G$.
By~\eqref{acdeck}, $x$ and $y$ have at most one common neighbor.
When $xy\notin E(G)$,
\begin{equation}\label{no-adj}
n\ge \bigl|{N[x]\cup N[y]\bigr|}\ge 2n-2\ell-1=n+(n-2\ell-1)\geq n.
\end{equation}
It follows that $n=2\ell+1$ and $d(x)=d(y)=n-\ell-1$
and $N[x]\cup N[y]=V(G)$.
If there is a third big vertex, then since $n-\ell-1\ge3$ it has at least two
neighbors in $N[x]$ or in $N[y]$, contradicting~\eqref{acdeck}. 
We conclude that in this case $\cD$ has exactly two star cards.

If three big vertices are pairwise adjacent, then they induce a $3$-cycle,
contradicting~\eqref{acdeck}.  Hence three big vertices must include
a nonadjacent pair, reducing to the previous case.

There remains only the case of exactly two big vertices $x$ and $y$, adjacent.
Now 
\begin{equation}\label{adj}
n\ge \bigl|N[x]\cup N[y]\bigr|\ge 2n-2\ell-2=n+(n-2\ell-2)\geq n-1.
\end{equation}
If $\bigl|N[x]\cup N[y]\bigr|=n-1$, then $x$ and $y$ both have degree $n-\ell-1$
and again there are exactly two star cards.  If $\bigl|N[x]\cup N[y]\bigr|=n$,
then either $n=2\ell+1$, with $x$ and $y$ having degrees $n-\ell-1$ and
$n-\ell$, or $n=2\ell+2$ with $x$ and $y$ both having degree $n-\ell-1$.
In the former case, there are $\ell+2$ star cards (since $n-\ell=\ell+1$)
and in the latter case there are two.


We have shown that $G$ can only have one or two big vertices, and if it has
two, then $\cD$ has either two or $\ell+2$ star cards, and in the last case
$n=2\ell+1$.  Since $n-\ell\ge4$, the degree $d$ of a big vertex is at least
$3$.  Hence $\CH d{n-\ell-1}$ cannot equal $2$ and cannot equal $\ell+2$ when
$n=2\ell+1$. This makes all cases distinguishable from the others.

We now know $m(T,G)$ for any reconstruction $G$ and every star $T$ with at
least $n-\ell$ vertices, and Lemma~\ref{counting} applies to yield the number
of vertices with degree $t$ for each $t$ at least $2$.  It remains to count
vertices with degree at most $1$.  Since we know the number $m$ of edges from
the $2$-deck, the number of vertices with degree $1$ is given by subtracting
the other known degrees from $2m$, and then the remaining vertices have degree
$0$.
\end{proof}

\begin{lemma}\label{diam}
Let $\cD$ be a deck having a connected card.  Every connected card has diameter
at least $2k+2$, and some connected card has diameter at most $2k+3$.
\end{lemma}
\begin{proof}
Since the deck is acyclic, every connected card is a tree.  A connected card
with diameter at most $2k+1$ would be a $j$-vine or $j$-evine with $j\le k$
having $n-\ell$ vertices, contradicting the definition of $k$.

For the second claim, let $C$ be a connected card.  If $C$ has diameter at
least $2k+3$, then $C$ contains a path with $2k+4$ vertices, which is a
$(k+1)$-evine.  By the definition of $k$, some $(k+1)$-evine or $(k+1)$-vine
$R$ has at least $n-\ell$ vertices.  Since $n-\ell\ge 2k+4$, we can iteratively
delete leaves outside a fixed longest path in $R$ to trim it to $n-\ell$
vertices.  We thus obtain a card that is a $(k+1)$-evine or $(k+1)$-vine, which
have diameter $2k+3$ or $2k+2$, respectively.  Hence some card has diameter
at most $2k+3$.
\end{proof}

When $n\ge2\ell+2$, we will show that $\cD$ cannot have both an acyclic
reconstruction $F$ and a nonacyclic reconstruction $H$ by showing that $H$
would have more $k$-centers or $(k+1)$-centers than $F$.  We next introduce
a tool for bounding the number of $j$-centers in a forest $F$.

\begin{definition}\label{markdef}
{\it The marking process.}
Let $z$ be a central vertex of a connected $(n-\ell)$-card $C$ with radius
$j+1$ in a forest $F$.  Let $Y$ be the set of neighbors of $z$ that lie on
paths of length $j+1$ in $C$ beginning at $z$, and let $d_C=\C Y$.  In the
component of $F$ containing $C$, every $j$-center $x$ that is not in $Y$
{\bf marks} one vertex $x'$ at distance $j$ from $x$ along a path that extends
the $z,x$-path in $F$ (such a vertex exists, since $x$ is a $j$-center).
\end{definition}

Note that $d_C$ is the maximum number of edge-disjoint paths of length $j+1$ in
$C$ with endpoint $z$.  In particular, $d_C=1$ when $C$ has diameter $2j+1$ and
$d_C\ge2$ when $C$ has diameter $2j+2$.  Figure~\ref{markfig} illustrates the
marking process for a card $C$ (in bold) within a tree $F$.  Here $C$ has
radius $3$ with center $z$, we have $j=2$ and $d_C=3$, with $Y=\{y_1,y_2,y_3\}$,
and $x_i$ marks $x'_i$.  Vertices of the form $x_i$, $y_i$, and $z$ are
$j$-centers in $F$.

\begin{figure}[h]
\begin{center}
\includegraphics[scale=0.5]{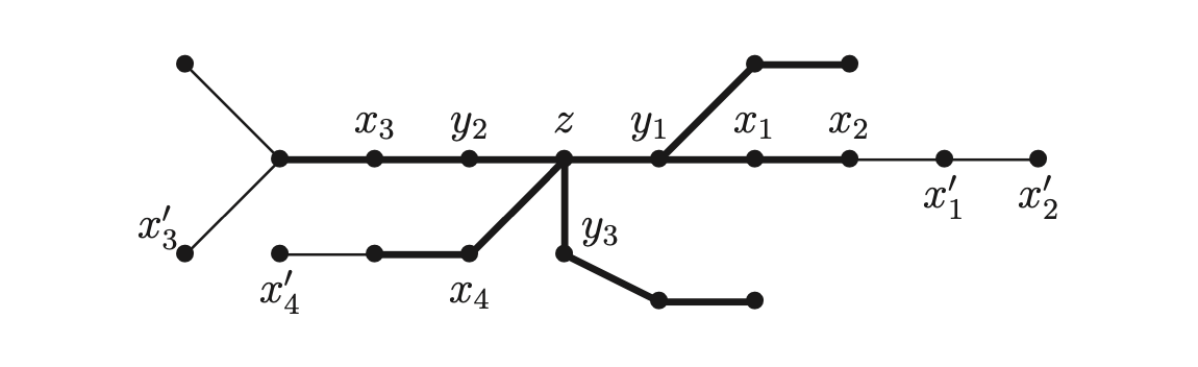}
\caption{The Marking Process}\label{markfig}
\end{center}
\end{figure}
\vspace{-.5pc}

\begin{lemma}\label{marking}
If $j\ge1$ and $C$ is a connected card with radius $j+1$ in the $(n-\ell)$-deck
$\cD$ of an $n$-vertex forest $F$, then the number of $j$-centers in $F$ is
at most $1+d_C+\ell$.  If equality holds, then in the marking process each
vertex of $F$ outside $C$ is marked and $F$ is a tree.
\end{lemma}
\begin{proof}
Let $F'$ be the component of $F$ containing $C$, and let $\ell'$ be the number
of vertices of $F'$ outside $C$.  All $j$-centers that are neighbors of $z$ on
paths of length $j+1$ from $z$ (and $z$ itself) do not mark vertices.  All
other $j$-centers in $F'$ mark a vertex that is outside $C$.  Since $F'$ has no
cycles, every vertex of $F'$ is marked by at most one $j$-center.  Thus $F$ has
at most $\ell'+d_C+1$ $j$-centers in $F'$.

There are $\ell-\ell'$ vertices of $F$ outside $F'$, and any that have degree
at most $1$ cannot be $j$-centers (since $j\ge1$).  Hence $F$ has at most
$1+d_C+\ell$ $j$-centers, with equality only if $F$ is a tree and all vertices
outside $C$ are marked.
\end{proof}

Note that the conclusion is false when $j=0$, since every vertex is a
$0$-center.

\section{Restricting to $n=2\ell+1$}\label{S2l2}

Given an acyclic $(n-\ell)$-deck $\cD$ for $n\ge2\ell+1$, with $k$ defined as
in Section~\ref{Stools}, we have proved that all $n$-vertex reconstructions
from $\cD$ have the same number of $k$-centers and have the same number of
$k$-central edges.  Our next aim is to prove this also for $(k+1)$-centers when
$\cD$ has no card with diameter $2k+2$.  We will need connected cards, which
are guaranteed when $\cD$ has reconstructions both with and without cycles.

\begin{definition}\label{ambig}
We say that a deck $\cD$ is {\it ambiguous} if it is the $(n-\ell)$-deck of
both an $n$-vertex acyclic graph $F$ and an $n$-vertex nonacyclic graph $H$.
\end{definition}

\begin{remark}\label{fullpath}
An ambiguous deck is acyclic, since $F$ contains no cycle.  Hence when $\cD$ is
ambiguous the graph $H$ has girth at least $n-\ell+1$, and thus $\cD$ has
connected cards (in particular, paths).
\end{remark}

\begin{lemma}\label{diam2k2}
If $n\ge2\ell+1$ and an ambiguous deck $\cD$ has no card with diameter $2k+2$,
then all reconstructions have the same number of $(k+1)$-centers.
\end{lemma}
\begin{proof}
Since $G$ has a $k$-evine and all $k$-evines have fewer than $n-\ell$ vertices,
we have $2k+2<n-\ell$.  By Remark~\ref{fullpath}, $\cD$ has a connected card.
Thus if no card has diameter $2k+2$, then some card has diameter $2k+3$, by
Lemma~\ref{diam}.  A card with diameter $2k+3$ is a $(k+1)$-evine.  In this
card are two $(k+1)$-vines, centered at the vertices of the central edge, so
$G$ has $(k+1)$-vines.

If any $(k+1)$-vine has at least $n-\ell$ vertices, then we obtain a
$(k+1)$-vine with $n-\ell$ vertices by iteratively deleting leaves outside
a fixed longest path unless $2k+3>n-\ell$, but this contradicts the inequality
$2k+2<n-\ell$.  However, a $(k+1)$-vine with $n-\ell$ vertices is a card
with diameter $2k+2$, which by hypothesis does not exist.  Hence all 
$(k+1)$-vines have fewer than $n-\ell$ vertices.

By Lemma~\ref{girth}, $G$ has girth at least $2k+4$.  Hence every $(k+1)$-vine
lies in a unique maximal $(k+1)$-vine, by Lemma~\ref{vinemax}.  Therefore,
the hypotheses of Lemma~\ref{counting} hold for the family of $(k+1)$-vines
in $G$, and the deck determines all the maximal $(k+1)$-vines in $G$, with
multiplicity (as in Corollary~\ref{countk} for $k$-vines).  The maximal
$(k+1)$-vines correspond to the $(k+1)$-centers, so we obtain the number
of $(k+1)$-centers.
\end{proof}

\begin{lemma}\label{C+Q}
Let $\cD$ be ambiguous, with $n\ge2\ell+1$ and $(n,\ell)\ne(5,2)$.  If $C$ is a
connected card in $\cD$ and $Q$ is a shortest cycle in $H$, then $C$ and $Q$
share at least four vertices and lie in a component of $H$ with at most $n-2$
vertices.  Furthermore, $C$ cannot be a star.
\end{lemma}
\begin{proof}
Since $\cD$ is acyclic, $Q$ has at least $n-\ell+1$ vertices.  Since $C$ has
$n-\ell$ vertices and $2n-2\ell+1>n$, subgraphs $C$ and $Q$ of $H$ intersect.
Since $F$ has at most $n-1$ edges and the component $H'$ of $H$ containing
$C$ and $Q$ has at least as many edges as vertices, $H'$ cannot be all of $H$.
If $H'$ has $n-1$ vertices, then $H$ has an isolated vertex and $F$ has at
least $n-1$ edges and is a tree, with no isolated vertices.  Since $F$ and $H$
have the same degree list (by Corollary~\ref{degrees}), $H'$ therefore has at
most $n-2$ vertices.  With $t=\C{V(C)\cap V(Q)}$, we have
$n-\ell+(n-\ell+1)-t\le n-2$, so $t\ge4$.

Since $Q$ is a shortest cycle, three vertices of $Q$ cannot have a common
neighbor in or outside $Q$.  Since $t\ge4$, we conclude that $C$ cannot be a
star.
\end{proof}

Henceforth let $\cD$ be the ambiguous $(n-\ell)$-deck of reconstructions $F$
and $H$ as in Definition~\ref{ambig}, and let $k$ be as in
Definition~\ref{defk}.
In the remainder of this section we restrict the possibility of ambiguous decks
to the case $n=2\ell+1$, which completes the $\ell$-recognizability proof when
$n\ge2\ell+2$.  We leave the boundary case $n=2\ell+1$ to the next section.

When we want to use the marking process to compare the numbers of $k$-centers,
we need to exclude the possibility $k=0$, since the conclusion of
Lemma~\ref{marking} is false when $j=0$.

\begin{lemma}\label{k=0}
Let $\cD$ be ambiguous with $n\ge2\ell+1\ge5$ and $(n,\ell)\ne(5,2)$.
If $k=0$, then $n=2\ell+1$ and a card with smallest diameter is a $1$-evine
that in a nonacyclic reconstruction intersects a cycle with $\ell+2$ vertices
exactly in four consecutive vertices.
\end{lemma}
\begin{proof}
Since $\cD$ is acyclic, $H$ has girth more than $n-\ell$.  Hence some card is
a path; it has at least four vertices, since $n-\ell\ge4$.  Thus any
reconstruction contains a $1$-evine, so $k=0$ requires a card that is a
$1$-evine or a $1$-vine.  By Lemma~\ref{C+Q}, no card is a $1$-vine (a star).

Hence a card $C$ with smallest diameter is a $1$-evine.  That is, $C$ is a tree
with two non-leaf vertices.  By Lemma~\ref{C+Q}, $C$ and a shortest
cycle $Q$ share at least four vertices in $H$.  Since $Q$ has at least five
vertices and is a shortest cycle in $H$, in $H$ there is no chord of $Q$ and no
vertex outside $Q$ with two neighbors in $Q$.  Hence each central vertex of
$C$ has only one neighbor that is in $Q$, so there are at most four vertices
of $C$ in $Q$, with equality only if they are two leaves and the two central
vertices of $C$ and occur consecutively along $Q$.

With $\C{V(C\cap Q)}=4$, we have $\C{V(C\cup Q)}\ge 2n-2\ell+1-4$.
Lemma~\ref{C+Q} then implies $2n-2\ell-3\le n-2$, which simplifies to
$n\le 2\ell+1$.  With $n=2\ell+1$, the cycle $Q$ must have exactly $\ell+2$
vertices.
\end{proof}

\begin{theorem}\label{2l+1}
Given $n\ge2\ell+1$ with $\ell\ge2$, let $\cD$ be an ambiguous deck with
acyclic reconstruction $F$ and nonacyclic reconstruction $H$ for which $k\ge1$,
and let $C$ be a card with minimum diameter in $\cD$.  These conditions require
$n=2\ell+1$, that $H$ has girth $\ell+2$, and that $F$ is a tree with exactly
$1+d_C+\ell$ $j$-centers, where $j$ is the radius of $C$.
\end{theorem}
\begin{proof}
By Lemma~\ref{diam}, no card has diameter less than $2k+2$, but some card
has diameter at most $2k+3$.  Let $Q$ be a shortest cycle in $H$, with length
$q$.  By Lemma~\ref{girth}, $q\ge2k+4$.

{\bf Case 1:}
{\it $C$ has diameter $2k+2$.}
By Corollary~\ref{countk}, $F$ and $H$ have the same number $s$ of $k$-central
edges.  Note that $C$ has radius $k+1$ and has $d_C$ $k$-central edges incident
to its unique center $z$.  These edges are also $k$-central in $F$.  An edge of
$F$ in the component containing $z$ is a $k$-central edge if and only if its
endpoint farther from $z$ is a $k$-center.  In other components, the number of
$k$-central edges is less than the number of $k$-centers.  Using $j=k$,
Lemma~\ref{marking} implies $s\le d_C+\ell$.

Among the $d_C$ $k$-central edges in $C$ incident to $z$, only two can lie in
$Q$.  Since $q\ge 2k+4$, every edge of $Q$ is a $k$-central edge in $H$.  Thus
$s\ge q+d_C-2$, and the bounds on $s$ yield $q\le\ell+2$.  With $q\ge n-\ell+1$,
we obtain $n\le2\ell+1$.  Since $n\ge2\ell+1$, we thus have $n=2\ell+1$ and
$q=\ell+2$.  The bounds on $s$ now yield $s= d_C+\ell$, so $F$ has exactly
$1+d_C+\ell$ $k$-centers, which by Lemma~\ref{marking} requires that $F$ is a
tree.

{\bf Case 2:}
{\it $C$ has diameter $2k+3$.}
By Lemma~\ref{diam2k2}, $F$ and $H$ have the same number $s'$ of
$(k+1)$-centers.  With diameter $2k+3$, $C$ has radius $k+2$ and two centers.
Let $z$ be a center in $C$.  By Lemma~\ref{marking} with $j=k+1$, we have
$s'\le 2+\ell$, since $d_{C}=1$.  Since $q\ge2k+4$, every vertex of $Q$ is a
$(k+1)$-center.  Hence $n-\ell+1\le q\le s'\le 2+\ell$, which simplifies to
$n\le 2\ell+1$.  Since $n\ge2\ell+1$, we have equality throughout, so
$n=2\ell+1$, and $q=2+\ell$, and $s'=2+\ell$, which by Lemma~\ref{marking}
requires that $F$ is a tree.
\end{proof}

\begin{corollary}\label{2l+2}
For $n\ge2\ell+2$, the family of $n$-vertex acyclic graphs is
$\ell$-recognizable.
\end{corollary}
\begin{proof}
By Theorem~\ref{2l+1}, an ambiguous deck can exist only when $n=2\ell+1$.
\end{proof}

\section{The Extreme Case $n=2\ell+1$}\label{S2l1}

The arguments of the previous section leave open the possibility of an
ambiguous deck when $n=2\ell+1$, and the graphs in Figure~\ref{52ex} yield
an ambiguous deck when $(n,\ell)=(5,2)$.  In this section we will prohibit
ambiguous decks when $n=2\ell+1$ and $\ell\ge3$, yielding the sharp threshold
on $n$ for $\ell$-recognizability of acyclicity.

Comparing the numbers of $k$-centers and $(k+1)$-centers in an acyclic and
a nonacyclic reconstruction only restricted us to $n\le 2\ell+1$.  Now that
we restrict to $n=2\ell+1$ and $\ell\ge3$.  We will distinguish these
possibilities by counting the cards that are paths (with $n-\ell$ vertices).
We will obtain a bound on this number for a special class of trees; this
is a result that may be of independent interest.  We will
eventually use the marking process to restrict the acyclic reconstruction to
this class when we have an ambiguous deck.

\begin{definition}
Fix the parameter $\ell$.  A {\it full path} in an $n$-vertex graph is a path
with $n-\ell$ vertices (a card in the $(n-\ell)$-deck).
A {\it branch vertex} in a tree is a vertex with degree at least $3$.
A {\it leg} of a non-path tree is a path in the tree whose endpoints
are a leaf and the branch vertex closest to it.
A {\it spider} is a tree with at most one branch vertex (trees with no branch
vertex are paths, in which we may designate any vertex as the ``root'' serving
the role of a branch vertex).  We denote a spider with legs of lengths
$\VEC m1d$ as $S_{\VEC m1d}$; it has $1+\SE i1d m_i$ vertices (an $n$-vertex
path can be described as $S_{m,n-1-m}$ for any $m$ with $1\le m\le n-2$.)
An $n$-vertex tree is {\it $\ell$-spiderly} if it contains a spider such that
all vertices not in the spider are within distance $(n-\ell-2)/2$ of the
branch vertex of the spider.
\end{definition}

Note that all spiders are $\ell$-spiderly.  The upper bound that we can prove
on the number of full paths in an $n$-vertex spider holds more generally for
all $\ell$-spiderly trees with $n$ vertices.

\begin{lemma}\label{spiderly}
For $n\ge2\ell+1\ge3$, every $\ell$-spiderly $n$-vertex tree contains at most
$\ell+3$ full paths, except for the spider $S_{1,1,1,1}$ when $\ell=2$.
\end{lemma}
\begin{proof}
An $\ell$-spiderly tree may allow many choices of the set $U$ inducing the
specified spider.  We thus view an instance as a pair $(T,U)$ and let
$\bU=V(T)-U$.  We consider a counterexample $(T,U)$ with smallest $n$
(over all $\ell$),
and with smallest $\C{\bU}$ among those minimizing $n$.  When $\ell=1$, an
$n$-vertex tree has at most two full paths, except that $S_{1,1,1}$ has three
full paths when $n=4$.  This is no problem, since $3<4=\ell+3$.
Hence we may assume $\ell\ge2$.

Let $z$ be the branch vertex of $T[U]$ (we may designate any vertex of $T[U]$
as $z$ when $z$ is a path, as long as the distance condition is satisfied for
vertices outside $U$).  By the minimality of $\C{\bU}$, neighbors of $z$ are in
$U$, and leaves of $T[U]$ are leaves of $T$.

Let $v$ be a leaf of $T$.  Since $(n-1)-(\ell-1)=n-\ell$, the distance
bound $(n-\ell-2)/2$ satisfied by vertices in $T$ also holds to make $T-v$
an $(\ell-1)$-spiderly tree on $n-1$ vertices.  If $T-v=S_{1,1,1,1}$ with the
forbidden parameters, then $T$ has parameters $(n,\ell)=(6,3)$, which violate
$n\ge2\ell+1$.  Therefore, we can apply the minimality of $n$ to conclude that
in $T-v$ there are at most $\ell+2$ full paths.  Full paths in $T$ not
containing $v$ are full paths in $T-v$, so if $v$ appears in at most one full
path we have the desired bound.  Hence we may assume that every leaf in $T$
appears in at least two full paths.

First consider the case $\bU=\nul$, where $T$ is a spider.  Since any leaf lies
in at least two full paths, $T$ has a branch vertex $z$.  Let $d$ be the degree
of $z$ in $T$, so $d\ge3$.  Let $v$ be the leaf in a shortest leg of $T$, with
length $a$.  If $d\ge4$, then two legs not containing $v$ must each have length
at least $n-\ell-1-a$, and some fourth leg with leaf $w$ has length at least
$a$.  Summing the lengths of these four legs yields $2n-2\ell-2\le n-1$, or
$n\le 2\ell+1$.  By the restriction to $n\ge2\ell+1$, equality holds, requiring
$T=S_{a,a,\ell-a,\ell-a}$ and $n-\ell=\ell+1$.  If $a<\ell-a$, then exactly
four full paths use $v$ or $w$ and a leg of length $\ell-a$, and $\ell-2a+1$
full paths use the two legs not containing $v$ or $w$.  The total is
$\ell-2a+5$, which is at most $\ell+3$ since $a\ge1$.  If $a=\ell-a$, then
there is also one full path from $v$ to $w$, but now exceeding $\ell+3$
requires $a=1$ and $\ell=2$, which occurs precisely for the exceptional case
$S_{1,1,1,1}$.

If $d=3$, then $T=S_{a,b,c}$.  To have each leaf in two full paths, the lengths
of any two legs sum to at least $n-\ell-1$.  The union of two legs together
having $t$ vertices contains $t-(n-\ell-1)$ full paths.  Hence the number of
full paths is $2(a+b+c)+3-3(n-\ell-1)$, which simplifies to $3\ell-n+4$.  Since
$n\ge2\ell+1$, the value is at most $\ell+3$.

Now we may assume $\bU\ne\nul$.  Let $v$ be a leaf of $T$ in $\bU$, and let $L$
be the leg of $T[U]$ closest to $v$ (since neighbors of $z$ are in $U$, $L$ is
well-defined).  Let $P$ be the path from $v$ to $z$.  The leg $L$ must be at
least as long as $P$, since otherwise we can enlarge $U$ to obtain an earlier
counterexample by replacing $L$ with $P$ in the spider without changing the
tree or its number of full paths.

Since vertices of $\bU$ are within distance $(n-\ell-2)/2$ of $z$ and full
paths have length $n-\ell-1$, the other end of any full path starting from
$v$ lies in $U$.  The bound on the length of $P$ implies that only one
of these paths can end on $L$, so they end on distinct legs in $T[U]$.  Let
$T^*$ be the tree obtained from $T-v$ by adding one leaf $v^*$ to extend $L$.
We obtain a new $\ell$-spiderly instance $(T^*,U^*)$, where $U^*=U\cup\{v^*\}$.
Since $T$ is a minimal counterexample, $T^*$ has at most $\ell+3$ full paths.
Suppose that $L$ has length at most $n-\ell-3$.  For any full path starting at
$v$ and ending on another leg $L'$, we instead have a full path starting at
$v^*$ that ends on $L'$ (since $L$ is at least as long as $P$).  Also $P\cup L$
may contain a full path starting at $v$, but it must end on $L$ after turning
away from $z$, and shifting the path to start closer to $z$ on $P$ and end
closer to $v^*$ replaces this path with one in $T^*$ that we have not yet
counted.  Thus $T$, like $T^*$, has at most $\ell+3$ full paths.

We conclude that for any leaf $v\in\bU$, the leg $L$ in $T[U]$ closest to
$v$ has at least $n-\ell-2$ edges.  Any full path starting at $v$ must use an
edge of $L$.  Therefore, if some full path $P'$ in $T$ shares no edges with
$L$, then counting the edges in $T$ at $v$, in $L$, and in $P'$ yields
$$
n-1\ge 1+(n-\ell-2)+(n-\ell-1)=2n-2\ell-2\ge n-1.
$$
where the last inequality uses $n\ge2\ell+1$.  Equality must hold throughout,
so $n=2\ell+1$, $L$ has $n-\ell-2$ edges, $v$ has distance $1$ from $L$, and
the only edge outside $P'\cup L$ is the edge $e$ at $v$.  Now $L$ is not long
enough to complete a full path starting at $v$.  Since the only edge outside
$P'\cup L$ is $e$, the path $P'$ contains some vertex $y$ of $L$.  If
$y\ne z$, then $V(P')-\{y\}\esub \bU$, but one end of $P'$ now has distance
at least $(n-\ell+1)/2$ from $z$.  Hence $P'$ contains $z$ and lies
in $T[U]$, but now one end of $P'$ is within distance $(n-\ell-1)/2$ of $z$,
too close to $v$ to finish a full path on that leg.  Thus at most one full path
starts at $v$, contradicting our earlier restriction.  Thus
every full path must share an edge with $L$.

Only one full path starting from $v$ can end in $L$.  To have two full paths
from $v$, the spider $T[U]$ needs another leg $L'$ where a full path from $v$
ends after passing through $z$.  Since $v$ has distance at most $(n-\ell-2)/2$
from $z$, the leg $L'$ has length at least $(n-\ell)/2$.  Since $T$ has no full
path edge-disjoint from $L$, any other leg $L^*$ in $T[U]$ has length at most
$(n-\ell-3)/2$.  Let $w$ be the leaf at the end of $L^*$.  Any full path from
$w$ shares an edge with $L$ and hence must travel to $z$ and then along $L$ to
reach length $n-\ell-1$.  Hence there can be only one such path, which
contradicts the need for $w$ to start two full paths.

We conclude that $T$ has no such leg $L^*$ in addition to $L'$.  This means 
that a second full path from $v$, besides the one ending on $L'$, must end
on $L$.  Thus $P\cup L$ has at least $n-\ell$ edges, since the edge of $L$
incident to $z$ is not in this full path.

Since all neighbors of $z$ lie in $U$, we are now restricted to degree $2$ at
$z$.  If $L'$ contains a branch vertex, leading to a leaf $v'$ outside $U$,
then the argument we gave for $v$ and $L$ also yields at least $n-\ell$ edges
in the union of $L'$ and the path from $v'$ to $z$.  Now $T$ has at least
$2n-2\ell$ edges, but with $n\ge2\ell+1$ this exceeds $n$.  Hence $L'$ contains
no branch vertex.  We can now shift $z$ to the first branch vertex along $L$ to
obtain an earlier instance $(T,U')$ with $U'$ augmented by a neighbor of that
branch vertex.

We have shown that there is no minimal counterexample.
\end{proof}

We have seen that when a card with smallest diameter in an ambiguous deck has
radius $j$, the reconstructions $F$ and $H$ may have the same number of
$j$-centers when $n=2\ell+1$, but that forces $F$ to be a tree.  In this
setting, we will forbid ambiguous decks by showing that $F$ and $H$ cannot have
the same number of full paths.  The marking process again forces $F$ to be a
tree, but to apply marking with $j=k$, again we must exclude the possibility
$k=0$.

\begin{lemma}\label{knot0}
If $\cD$ is ambiguous and $n=2\ell+1$ with $\ell\ge3$, then $k\ge1$.
\end{lemma}
\begin{proof}
Suppose $k=0$.  By Lemma~\ref{k=0}, a card $C$ with smallest diameter is
a $1$-evine (double-star) that in a nonacyclic reconstruction $H$ intersects
a shortest cycle $Q$ exactly in four consecutive vertices, and $Q$ has length
$\ell+2$.

By Lemma~\ref{C+Q}, the component $H'$ of $H$ containing $C\cup Q$ has at most
$n-2$ vertices.  With $C$ having $n-\ell$ vertices and $Q$ having $\ell+2$
vertices, $H'$ has at least $n-2$ vertices, so equality holds.  Thus
$H'$ consists of an $(\ell+2)$-cycle plus $\ell-3$ pendant edges at two
consecutive vertices.  The cards of $H$ that are paths (with $\ell+1$ vertices)
lie along $Q$ or start at a leaf of $C$ not in $Q$.  There are $\ell+2$ such
paths along $Q$ and $2(\ell-3)$ that start at leaves of $C$ not in $Q$, for a
total of $3\ell-4$ cards that are paths.

In $H'$ we have $n-2$ edges.  An $n$-vertex forest with $n-2$ edges cannot have
two isolated vertices.  Since $H$ and $F$ have the same degree list (by
Corollary~\ref{degrees}), $H$ has at most one isolated vertex.  Hence the two
vertices of $H$ outside $H'$ must be adjacent, giving both graphs $n-1$ edges.
Hence $F$ is a tree.  

The vertices of $H$ not in $C$ all have degree $2$, except for the two leaves
outside $H'$.  Since $F$ and $H$ have the same degree list, the tree $F$ grows
from $C$ only by appending edges at leaves to extend paths.  That is, the
central vertices of $C$ cannot receive more incident edges, and no additional
branch vertices can be created.

The $1$-evine $C$ may have only one branch vertex and be obtained from a star by
appending one edge.  In this case $C$ and $F$ are both spiders and $F$ has at
most $\ell+3$ full paths.  Otherwise, $C$ and $F$ have exactly two branch
vertices, and they are adjacent.  If at either of the two branch vertices of
$F$ there is at most one leg with length at least $(n-\ell-3)/2$, then $F$ is
$\ell$-spiderly, since outside the largest spider of degree $3$ with legs
emanating from the other branch vertex $z$ are paths that reach distance
at most $(n-\ell-2)/2$ from $z$.

Otherwise, from the two branch vertices of $F$ there are four legs that
each have at least $\CL{(n-\ell-3)/2}$ edges.  Note that $n-\ell$ has
opposite parity from $\ell$, so these legs have at least $(n-\ell-2)/2$
edges when $\ell$ is odd, $(n-\ell-3)/2$ when $\ell$ is even.

When $\ell$ is odd, these four legs plus the central edge occupy
$2n-2\ell-3$ edges.  Since $2n-2\ell-3=n-2$, there remains only one edge to
add.  The subtree before adding that edge has four full paths, each consisting
of a leg from each branch vertex plus the central edge.  Adding one more
edge creates at most three more full paths, achieved by extending one leg.
Hence $F$ has fewer than $\ell+3$ full paths when $\ell\ge5$.

When $\ell$ is even, the four legs of length $(n-\ell-3)/2$ plus central edge
occupy $n-4$ edges but do not create any full path.  With only three edges
to add, only $2+3+3$ full paths can be created, so $F$ has fewer than
$\ell+3$ full paths when $\ell\ge6$.  When $\ell=4$, in fact only seven full
paths can be created, because when one of the legs is extended by two edges,
the leaf will be too far away to start a full path that reaches past the
opposite branch vertex.

We have shown that $F$ has at most $\ell+3$ full paths when $\ell\ge4$,
but we found $3\ell-4$ full paths in $H$.  Since $3\ell-4>\ell+3$ when
$\ell\ge4$, there is no ambiguous $\cD$ with $k=0$ unless $\ell=3$.
In that case $\ell-3=0$ and $n-\ell=4$, and $C$ is just a $4$-vertex path.
We find that $F$ is $P_7$ and contains four copies of $P_4$, while $H$ is the
disjoint union of a $5$-cycle and an edge, containing five copies of $P_4$.
\end{proof}

\begin{lemma}\label{fullpaths}
Given $n=2\ell+1$, let $\cD$ be an ambiguous deck with acyclic reconstruction
$F$ and nonacyclic reconstruction $H$.  The graph $F$ is a
tree that has at most $\ell+3$ full paths.
\end{lemma}
\begin{proof}
Let $C$ be a card with minimum diameter; the diameter of $C$ is $2k+2$ or
$2k+3$, by Lemma~\ref{diam}, with radius $k+1$ or $k+2$, respectively.  Let $C$
have radius $j+1$, so $j\in\{k,k+1\}$.  By Theorem~\ref{2l+1}, $F$ has
$1+d_C+\ell$ $j$-centers and is a tree.  That is, the marking process from a
center $z$ of $C$ marks all $\ell$ vertices of $F$ outside $C$.

We claim that $F$ cannot have vertices $v$ and $v'$ outside $C$ at equal
distance from $z$ whose paths to $z$ share an edge.  Consider such a pair
closest to $z$.  Their paths to $z$ cannot meet after traveling at most $j$
steps toward $z$, because then there would be only one $j$-center that can mark
them both, and it marks only one vertex.  If the paths meet after traveling
more than $j$ steps toward $z$, then since $z$ and its neighbors mark no
vertices, $v$ and $v'$ have distance at least $j+3$ from $z$.  Now the vertices
next to $v$ and $v'$ on the paths to $z$ are outside $C$ and form such a pair
closer to $z$.

We conclude that the subtree of $F$ induced by all the vertices outside $C$
and the vertices on their paths to $z$ is a spider.  That is, $F$ grows from
$C$ only by extending edge-disjoint paths from $z$.  If $C$ has at most one
branch vertex, then $F$ is a spider, and by Lemma~\ref{spiderly} $F$ has at
most $\ell+3$ full paths.

If $C$ has at least two branch vertices, then any longest path in $C$ has at
most $n-\ell-2$ vertices.  The longest path in $C$ has $2k+3$ or $2k+4$
vertices, depending on the diameter.  In either case, with radius $j+1$ for
$C$, we obtain $j+1\le (n-\ell-2)/2$.  All vertices of $F$ not in the spider we
have constructed are in $C$ and hence are within distance $(n-\ell-2)/2$ of
$z$.  This makes $F$ an $\ell$-spiderly tree, and by Lemma~\ref{spiderly} it
has at most $\ell+3$ full paths.
\end{proof}

\begin{theorem}\label{main}
For $n\ge 2\ell+1\ge3$, the family of $n$-vertex acyclic graphs is
$\ell$-recognizable, except when $(n,\ell)=(5,2)$.
\end{theorem}
\begin{proof}
Suppose that there is an ambiguous deck $\cD$ with reconstructions $F$ and $H$
as we have been discussing.  By Theorem~\ref{2l+1}, $n=2\ell+1$ and $F$ is a
tree and $H$ has girth $\ell+2$.  By Lemma~\ref{fullpaths}, $F$ has at most
$\ell+3$ full paths.  Hence at most $\ell+3$ cards in the deck are paths;
this is the key point that will yield a contradiction.

With $(n,\ell)=(5,2)$ excluded and the result already known for $\ell=1$
by the original work of Kelly~\cite{Kel2}, we may assume $n-\ell=\ell+1\ge4$.
Full paths have length $\ell$.

By Theorem~\ref{2l+1}, a shortest cycle $Q$ in $H$ has length $\ell+2$.
There are $\ell+2$ full paths in $Q$.  Let $C$ be a card with smallest diameter.
By Lemma~\ref{C+Q}, $C$ and $Q$ lie in the same component of $H$.
Since this component has at most $n-2$ vertices, $C$ has at most $\ell-3$
vertices outside $Q$, and hence each vertex of $C$ is within distance $\ell-3$
of $V(Q)$.  Therefore, from each vertex of $C$ outside $V(Q)$, one can travel
to $V(Q)$ in $C$ and then complete a full path in either direction along $Q$.
Hence $H$ contains at least $\ell+2+2t$ full paths, where $t$ is the number of
vertices of $C$ outside $Q$.

If $t\ge1$, then this contradicts the previous conclusion that the deck has at
most $\ell+3$ cards that are paths.  Since $Q$ has no chords, we can only have
$t=0$ if $C$ is a path contained in $Q$.  If the smallest diameter card in
$\cD$ is a path, then $F$ is a path.  In that case $F$ contains only $\ell+1$
full paths, which again is fewer than the $\ell+2$ full paths in $H$.
\end{proof}

\bigskip
\centerline{\bf\large Acknowledgment}
We thank the referees for their careful reading and many helpful comments.

\end{document}